\documentclass[11pt,twoside]{amsart}
\usepackage{amssymb, amsmath, enumerate, palatino, hyperref}
\usepackage{bbding}
\usepackage{mathdots}
\usepackage[normalem]{ulem}
\usepackage{fullpage}
\usepackage[T1]{fontenc}
\usepackage{float}

\usepackage{mathrsfs,manfnt,array,graphicx,tikz,tkz-euclide,verbatim,wasysym}
\usepackage{aliascnt}
\usepackage{hyperref}
 \definecolor{dark-red}{rgb}{0.4,0.15,0.15}
   \definecolor{dark-blue}{rgb}{0.15,0.15,0.4}
 \definecolor{medium-blue}{rgb}{0,0,0.5}
\hypersetup{
    colorlinks, linkcolor=dark-red,
    citecolor=dark-blue, urlcolor=medium-blue
}
\usepackage{mathrsfs}
\usepackage[
textwidth=3cm,
textsize=small,
colorinlistoftodos]
{todonotes}

\usepackage{tikz}
\usepackage{tikz-cd}

\usepackage[T1]{fontenc}

\numberwithin{equation}{section} %Fiddles with numbering system of the following.

\theoremstyle{plain}

\newaliascnt{theorem}{equation}  
\newtheorem{theorem}[theorem]{Theorem}  
\aliascntresetthe{theorem}

 \theoremstyle{definition}

\newaliascnt{prop}{equation}  
\newtheorem{prop}[prop]{Proposition}
\aliascntresetthe{prop}

\newaliascnt{lemma}{equation}  
\newtheorem{lemma}[lemma]{Lemma}
\aliascntresetthe{lemma}

\newaliascnt{corollary}{equation}  
\newtheorem{corollary}[corollary]{Corollary}
\aliascntresetthe{corollary}

\newaliascnt{claim}{equation}  

\aliascntresetthe{claim}

\newaliascnt{conjecture}{equation}  
\newtheorem{conjecture}[conjecture]{Conjecture}
\aliascntresetthe{conjecture}

\newaliascnt{question}{equation}  

\aliascntresetthe{question}

\newaliascnt{definition}{equation}  
\newtheorem{definition}[definition]{Definition}
\aliascntresetthe{definition}

\newaliascnt{construction}{equation}  
\newtheorem{construction}[construction]{Construction}
\aliascntresetthe{construction}

\newaliascnt{notation}{equation}  
\newtheorem{notation}[notation]{Notation}
\aliascntresetthe{notation}

\newaliascnt{example}{equation}  
\newtheorem{example}[example]{Example}
\aliascntresetthe{example}

\newaliascnt{exercise}{equation}  

\aliascntresetthe{exercise}

\theoremstyle{remark}

\newaliascnt{remark}{equation}  
\newtheorem{remark}[remark]{Remark}
\aliascntresetthe{remark}

\newaliascnt{convention}{equation}  

\aliascntresetthe{convention}

\theoremstyle{plain}

\newcommand{\ZZ}{\mathbb Z}
\newcommand{\QQ}{\mathbb Q}

\newcommand{\ff}{\mathscr F}

\newcommand{\Sub}{\operatorname{Sub}}

\newcommand{\Tr}{\operatorname{Tr}}

\newcommand{\STr}{\operatorname{STr}}
\newcommand{\SCov}{\operatorname{SCov}}
\newcommand{\stirling}[2]{\begin{Bmatrix}#1\\#2\end{Bmatrix}}
\renewcommand{\mod}{\text{ mod }}
\newcommand{\cP}{\mathcal{P}}
\newcommand{\cO}{\mathscr{O}}
\newcommand{\Ho}{\operatorname{Ho}}
\newcommand{\NOp}{N_\infty\text{-}\mathbf{Op}}
\newcommand{\cL}{\mathcal{L}}
\newcommand{\cI}{\mathcal{I}}

\title{Saturated and linear isometric transfer systems \\ for cyclic groups of order $p^mq^n$}

\author{Usman Hafeez}
\address{Department of Mathematics, Reed College, Portland, OR 97202, USA}
\email{usman-20@live.com}

\author{Peter Marcus}
\address{Department of Mathematics, Tulane University, New Orleans, LA 70118, USA}
\email{pmarcus1@tulane.edu}

\author{Kyle Ormsby}
\address{Department of Mathematics, Reed College, Portland, OR 97202, USA}
\email{ormsbyk@reed.edu}

\author{Ang\'{e}lica M. Osorno}
\address{Department of Mathematics, Reed College, Portland, OR 97202, USA}
\email{aosorno@reed.edu}

\begin{document}

\begin{abstract}
Transfer systems are combinatorial objects which classify $N_\infty$ operads up to homotopy.  By results of A.~Blumberg and M.~Hill \cite{BH}, every transfer system associated to a linear isometries operad is also \emph{saturated} (closed under a particular two-out-of-three property).  We investigate saturated and linear isometric transfer systems with equivariance group $C_{p^mq^n}$, the cyclic group of order $p^mq^n$ for $p,q$ distinct primes and $m,n\ge 0$.  We give a complete enumeration of saturated transfer systems for $C_{p^mq^n}$. We also prove J.~Rubin's \emph{saturation conjecture} for $C_{pq^n}$; this says that every saturated transfer system is realized by a linear isometries operad for $p,q$ sufficiently large (greater than $3$ in this case).
\end{abstract}

\maketitle

\section{Introduction} 
Fix a finite group $G$.  Highly structured commutative $G$-equivariant ring spectra can support multiplicative norm maps associated with a class of finite $H$-sets, $H$ ranging through subgroups of $G$.  The $G$-$N_\infty$ operads of A.~Blumberg and M.~Hill \cite{BH} parametrize ring structures with such an admissible family of norms.  Following the work of \cite{BH, BP, GW, rubin_comb}, J.~Rubin \cite{rubin} and S.~Balchin, D.~Barnes, and C.~Roitzheim \cite{BBR} independently prove that the homotopy category of $G$-$N_\infty$ operads is equivalent to the combinatorially-defined category of $G$-transfer systems (see \autoref{operad-transfer}). The structure of the lattice of transfer systems on an Abelian group was recently explored in \cite{selfdual}.

The transfer systems induced by certain natural families of $G$-$N_\infty$ operads have additional special properties.  In particular, any transfer system realized by an equivariant linear isometries operad (see \autoref{linear_isometries}) is \emph{saturated} (see \autoref{saturated}).  It is not the case, though, that every saturated transfer system arises in this fashion, as shown in \cite{rubin}.  This raises two fundamental questions which we address in this paper:
\begin{enumerate}[(1)] 
\item For a given group $G$, how many saturated $G$-transfer systems exist?
\item For a given group $G$, which saturated $G$-transfer systems can be realized by a linear isometries operad?
\end{enumerate} 
We say that a (necessarily saturated) transfer system realized by a linear isometries operad is \emph{linear isometric}; thus the second question may be rephrased as asking ``Which saturated $G$-transfer systems are linear isometric?''

Fix $p,q$ distinct primes and let $C_{p^mq^n}$ denote the cyclic group of order $p^mq^n$, for $m,n\ge 0$.  Let $s(m,n)$ denote the number of saturated $C_{p^mq^n}$-transfer systems.\footnote{After the definition is presented, it will be clear that this number is independent of $p$ and $q$.} We provide the following answers to the above questions for $G = C_{p^mq^n}$:

\begin{theorem}[see \autoref{closed form} and \autoref{exp gen}]
For all $m,n\geq 0$,
\[
 s(m,n)=\sum_{j=2}^{m+2}(-1)^{m-j}  \stirling{m+1}{j-1}\frac{j!}{2}j^n,
\]
where $\stirling{r}{s}$ denotes the Stirling number of the second kind enumerating $s$-block partitions of a set of cardinality $r$. Furthermore, the exponential generating function for $s(m,n)$ takes the form
\[
  \sum_{m,n\ge 0}\frac{s(m,n)}{m!n!}x^my^n = \frac{e^{2x+2y}}{(e^x+e^y-e^{x+y})^3}.
\]
\end{theorem}

\begin{theorem}[see \autoref{Cpqn}]
Let $n\geq 0$. For $G = C_{pq^n}$, and $p,q>3$, all saturated transfer systems are linear isometric.
\end{theorem}

The latter theorem verifies an instance of Rubin's saturation conjecture, which loosely says that every saturated transfer system for a cyclic group of order $n$ is linear isometric as long as the prime divisors of $n$ are sufficiently large; see \autoref{sat conj} for a precise statement.

\begin{remark}
E.~Franchere, the third and fourth authors of this paper, W.~Qin, and R.~Waugh expose a surprising connection between transfer systems and weak factorization systems (in the sense of abstract homotopy theory) in \cite{selfdual}. Under this correspondence, it turns out that saturated transfer systems are in bijection with model structures on the poset category $\Sub(G)$ of subgroups of $G$ for which all morphisms are fibrations. As such, \autoref{closed form} and \autoref{exp gen} give a complete enumeration of such model structures as well. The details of the relation between transfer systems and model structures are explained in \cite{boor}.
\end{remark}

\subsection*{Organization}
In \autoref{sec:trans}, we recall the definitions of and fundamental theorems regarding $N_\infty$ operads and (saturated) transfer systems.  In \autoref{sec:sat}, we begin our study of saturated transfer systems on $C_{p^mq^n}$ and prove some structural results about these objects. In particular, we reduce their study to a combinatorial game on the $m\times n$ grid.  We use this description in \autoref{sec:enum} to complete our enumeration of saturated transfer systems on $C_{p^mq^n}$. This takes three forms: a recurrence, a closed formula, and an exponential generating function.  Finally, in \autoref{sec:sat_conj} we prove the saturation conjecture for $C_{pq^n}$.

\subsection*{Acknoledgments}
We thank J.~Rubin for suggesting this problem and generously sharing his expertise. We thank I.~Kriz for discovering and sharing the exponential generating function of \autoref{exp gen} after the recurrence \autoref{recursive} was presented in a talk by the third named author. We also thank F.~Castillo for sharing combinatorial perspectives and expertise. This research was completed as part of the 2020 Electronic Collaborative Mathematics Research Group (eCMRG), led by the third and fourth authors and  supported by NSF grant DMS-1709302; we thank the other eCMRG partipants E.~Franchere, W.~Qin, and R.~Waugh for numerous useful conversations.

\section{$N_\infty$ operads and transfer systems}\label{sec:trans}

\subsection{$N_\infty$ operads}
In order to frame our work on transfer systems, we need to recall Blumberg--Hill's notion of an $N_\infty$ operad \cite{BH}, paying special attention to the example of linear isometry operads. We assume that the reader is familiar with the basic theory of (symmetric) operads. For $n\ge 0$, we denote the symmetric group on $n$ letters by $\mathfrak S_n$.

\begin{definition}
A \emph{$G$-$N_\infty$ operad} is a symmetric operad $\cO$ in $G$-spaces satisfying the following three properties:
\begin{itemize}
\item for all $n\ge 0$, the $G\times \mathfrak S_n$-space $\cO(n)$ is $\mathfrak S_n$-free,
\item for every $\Gamma\le G\times \mathfrak S_n$, the $\Gamma$-fixed point space $\cO(n)^\Gamma$ is empty or contractible, and
\item for all $n\ge 0$, $\cO(n)^G$ is nonempty.
\end{itemize}
A \emph{map of $G$-$N_\infty$ operads} $\varphi\colon \cO_1\to\cO_2$ is a morphism of operads in $G$-spaces.  We denote the associated category of $G$-$N_\infty$ operads by $\NOp^G$.

For a map $\varphi\colon \cO_1\to\cO_2$ of $G$-$N_\infty$ operads, the map at level $n$ is in particular $G\times \mathfrak S_n$-equivariant. We say that $\varphi$ is a \emph{weak equivalence} if $\varphi\colon \cO_1(n)^\Gamma\to \cO_2(n)^\Gamma$ is a weak homotopy equivalence of topological spaces for all $n\ge 0$ and $\Gamma\le G\times \mathfrak S_n$. The associated homotopy category (formed by inverting weak equivalences) is denoted $\Ho(\NOp^G)$.
\end{definition}

Note that $N_\infty$ operads are, in particular, nonequivariant $E_\infty$ operads, and thus parametrize operations that are associative and commutative up to higher homotopies.  Additionally, $N_\infty$ operads \emph{admit norms} for particular finite $H$-sets, $H\le G$ in the following sense.  Given an $H$-set $T$, let $\Gamma(T)\le G\times \mathfrak S_{|T|}$ denote the graph of a permutation representation of $T$.

\begin{definition}
A $G$-$N_\infty$ operad $\cO$ \emph{admits norms} for a finite $H$-set $T$ when $\cO(|T|)^{\Gamma(T)}$ is nonempty.
\end{definition}

Particular examples of $N_\infty$ operads include linear isometries operads defined on $G$-universes \cite{LMS}. We recall the definition.

\begin{definition}\label{linear_isometries}
A \emph{$G$-universe} $U$ is a countably infinite-dimensional real $G$-inner product space such that it contains each finite-dimensional subrepresentation infinitely often and contains the trivial representation.
The \emph{linear isometries operad} $\cL(U)$ is given at level $n$ by the space $\cL(U^n,U)$ of all (not necessarily equivariant) linear isometries, with $G$ acting by conjugation, and $\mathfrak S_n$ acting by permuting inputs. The operadic composition is given by composition of isometries.
\end{definition}

\subsection{Transfer systems}
A $G$-transfer system is a combinatorial object defined as a particular sub-poset of $\Sub(G)$, the subgroup lattice of $G$.  For $H\le G$ and $g\in G$, let ${}^gH = gHg^{-1}$ denote the $g$-conjugate of $H$.

\begin{definition}
A \emph{$G$-transfer system} is a relation $\to$ on $\Sub(G)$ that refines the inclusion relation\footnote{This means that $K\to H$ implies $K\le H$.} and satisfies the following properties:
\begin{itemize}
    \item (reflexivity) $H \to H$ for all $H\leq G$,
    \item (transitivity) $K \to H$ and $L \to K$ implies $L \to H$,
    \item (closed under conjugation) $K \to H$ implies that ${}^gK \to {}^gH$ for all $g\in G$,
    \item (closed under restriction) $K \to H$ and $M\leq H$ implies $(K \cap M) \to M$.
\end{itemize}
We denote the collection of transfer systems by $\Tr(G)$ and view it as a poset under the refinement relation.
\end{definition}

In other words, a transfer system is a sub-poset of $\Sub(G)$ where the relation is closed under conjugation and restriction.  Note that conjugation is trivial when $G$ is Abelian; in this case transfer systems only depend on the lattice structure of $\Sub(G)$.  Saturated transfer systems have an additional two-out-of-three property:
\begin{definition}\label{saturated}
A $G$-transfer system $\to$ is \emph{saturated} if it additionally satisfies the following property:
\begin{itemize}
\item (two-out-of-three) if $L\le K\le H\le G$ and two of the three relations $L\to K$, $L\to H$, $K\to H$ hold, then so does the third.
\end{itemize}
We denote the collection of saturated transfer systems by $\STr(G)$.
\end{definition}

By transitivity and closure under restriction, the two-out-of-three property may be rephrased as follows:
\[
  \text{if }L\le K\le H\text{ and }L\to H\text{, then } K\to H.
\]

The link between $N_\infty$ operads and transfer systems is provided by the following construction.  Given a $G$-$N_\infty$ operad $\cO$, define the relation $\to_\cO$ by the rule
\[
  K\to_{\cO}H\text{ if and only if }K\le H\text{ and }\cO([H:K])^{\Gamma(H/K)}\ne \varnothing
\]
where $\Gamma(H/K)$ is the graph of some permutation representation $H\to \mathfrak S_{[H:K]}$ of $H/K$.

\begin{theorem}\label{operad-transfer}
The assignment
\[
\begin{aligned}
  \NOp^G&\longrightarrow \Tr(G)\\
  \cO&\longmapsto (\to_\cO)
\end{aligned}
\]
induces an equivalence
\[
  \Ho(\NOp^G)\simeq \Tr(G)
\]
(considering the poset $\Tr(G)$ as a category). Moreover, if $\cO$ is a linear isometries operad, then $\to_\cO$ is saturated.
\end{theorem}

\begin{remark}
 In \cite{BH}, Blumberg and Hill defined $G$-indexing systems, which are collections of finite $H$-sets for varying subgroups $H$ of $G$ satisfying certain properties. These collections form a poset $\cI(G)$ under inclusion. They proved that every $N_\infty$ operad $\cO$ gives rise to an indexing system, and that this assignment gives a functor that descends to the homotopy category, with the resulting functor being full and faithful. Blumberg and Hill further conjectured that the functor is surjective, which was established independently by P.~Bonventre and L.~Pereira \cite{BP}, J.~Guti\'errez and D.~White \cite{GW}, and Rubin \cite{rubin_comb}. Both Rubin \cite{rubin} and Balchin, Barnes, and Roitzheim \cite{BBR} proved that the poset $\cI(G)$ of indexing systems is isomorphic to the poset $\Tr(G)$ of transfer systems. The result about linear isometries operads is the translation of the corresponding statement for indexing systems (cf. \cite[p. 678]{BH}) into the language of transfer systems (cf. \cite[Theorem 3.7]{rubin}).
\end{remark}

\autoref{operad-transfer} gives the precise link between $N_\infty$ operads and transfer systems, and explains the inclusion
\[
  \{\text{linear isometric $G$-transfer systems}\} \subseteq \STr(G).
\]
As Rubin points out in \cite[\S5.1]{rubin}, it was initially expected that the reverse inclusion would hold as well. As examples in \emph{op.cit} show, this is not generally true. Nonetheless, Rubin conjectures that when $|G|$ has large prime divisors, every saturated transfer system is linear isometric:

    \begin{conjecture}[Rubin]\label{sat conj}
        Fix a sequence of positive integers $r_1,\ldots,r_k$.  Then for distinct sufficiently large primes  $p_1,\ldots,p_k$, every saturated transfer system on $C_{p_1^{r_1}\cdots p_k^{r_k}}$ can be realized by a linear isometries operad.
    \end{conjecture}

    Rubin verified the conjecture for $C_{p^n}$ for all $n\geq 1$ and for $C_{pq}$. In \autoref{sec:sat_conj}, we verify the conjecture for $C_{pq^n}$ for all $n\geq 1$.

\subsection{Generating (saturated) transfer systems} 
Recall that for a poset $P$, we say that $x<y$ is a \emph{cover relation} if there is no $z\in P$ such that $x<z<y$. The saturation property implies that saturated transfer systems can be described in terms of the cover relations it contains. 

\begin{definition}
Let $R$ be a binary relation on $\Sub(G)$ that refines inclusion. The transfer system \emph{generated} by $R$, denoted by $\langle R \rangle$, is the minimal transfer system that contains $R$. An explicit description can be found in \cite[Construction A.1]{rubin}. 
\end{definition}

\begin{prop}[{\cite[Proposition 5.8]{rubin}}]\label{generatedcover}
        \label{irreducible}
        Let $\to$ be a saturated $G$-transfer system.
        Then $\to$ is generated  by the relation 
        \[\{(K,H) \mid K \to H \text{ and } (K,H) 
        \text{ is a cover relation in }\Sub(G)\}.\]
        \end{prop}
        
        As a result, a saturated $G$-transfer system is uniquely determined by the cover relations in $\Sub(G)$ it contains.
        
         \begin{remark}
        Note that a general $G$-transfer system is not necessarily
        generated by a set of cover relations, as the following example (``the chickenfoot'') illustrates for $G=C_{pq}$.
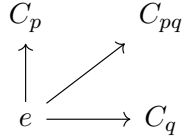
\[\begin{tikzcd}
  {C_p} & {C_{pq}} \\
  e & {C_q}
  \arrow[from=2-1, to=1-1]
  \arrow[from=2-1, to=2-2]
  \arrow[from=2-1, to=1-2]
\end{tikzcd}\]
            
    \end{remark}

\section{Saturated transfer systems on $C_{p^mq^n}$}\label{sec:sat}

In this section we concentrate on studying transfer systems on the group $C_{p^mq^n}$. In what follows, for $k\geq 0$, we denote by $[k]$ the poset $\{0<1<\dots< k\}$.

The subgroup lattice $\Sub(C_{p^mq^n})$ is isomorphic to the grid $[m]\times [n]$, with the subgroup $C_{p^iq^j}$ corresponding to $(i,j)$. For ease of notation, we will use this identification when referring to transfer systems on $C_{p^mq^n}$.
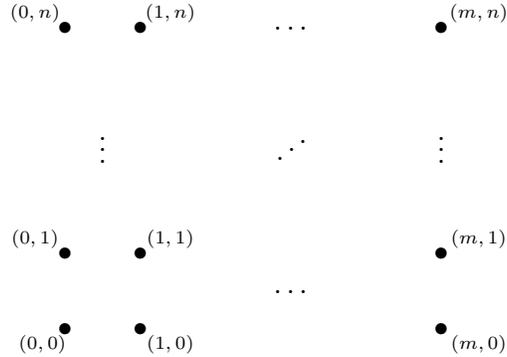
\begin{center}
\begin{tikzpicture}
\node at (0,0) {$\bullet$};
\node at (-0.3, -0.2) {\tiny{$(0,0)$}};
\node at (0,1) {$\bullet$};
\node at (-0.4, 1.2) {\tiny{$(0,1)$}};
\node at (1,0) {$\bullet$};
\node at (1.4,-0.2) {\tiny{$(1,0)$}};
\node at (1,1) {$\bullet$};
\node at (1.4,1.2) {\tiny{$(1,1)$}};
\node at (5,0) {$\bullet$};
\node at (5.5, -0.2) {\tiny{$(m,0)$}};
\node at (5,4) {$\bullet$};
\node at (5.5,4.2) {\tiny{$(m,n)$}};
\node at (0,4) {$\bullet$};
\node at (-0.4,4.2) {\tiny{$(0,n)$}};
\node at (5,1) {$\bullet$};
\node at (5.5,1.2) {\tiny{$(m,1)$}};
\node at (1,4) {$\bullet$};
\node at (1.4,4.2) {\tiny{$(1,n)$}};
\node at (0.5,2.5) {$\vdots$};
\node at (3,0.5) {$\hdots$};
\node at (5,2.5) {$\vdots$};
\node at (3,4) {$\hdots$};
\node at (3,2.5) {$\iddots$};
\end{tikzpicture}
\end{center}

            The following is meant to clarify how we will define and refer to rows and columns of the grid.
            
    \begin{notation}
        In a $m\times n$ grid, \emph{row} $j$ will refer to the edges between $(i,j-1)$ and $(i,j)$ for $i=0,\dots , m$. As such, there are $n$ \emph{rows} in the grid, numbered from 1 to $n$ (there is no row 0). Row $j$ is shown below.
        
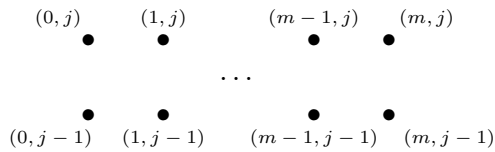
\begin{center}
            \begin{tikzpicture}
            \node at (0,0) {$\bullet$};
            \node at (0,1) {$\bullet$};
            \node at (1,0) {$\bullet$};
            \node at (1,1) {$\bullet$};
            \node at (3,0) {$\bullet$};
            \node at (3,1) {$\bullet$};
            \node at (4,0) {$\bullet$};
            \node at (4,1) {$\bullet$};
            \node at (-0.5,-0.3) {\tiny{$(0,j-1)$}};
            \node at (1,-0.3) {\tiny{$(1,j-1)$}};
            \node at (-0.4,1+0.3) {\tiny{$(0,j)$}};
            \node at (1,1+0.3) {\tiny{$(1,j)$}};
            \node at (2,0.5) {$\dots$};
            \node at (3, -0.3) {\tiny{$(m-1,j-1)$}};
            \node at (4.8, -0.3) {\tiny{$(m,j-1)$}};
            \node at (3, 1.3) {\tiny{$(m-1,j)$}};
            \node at (4.5, 1.3) {\tiny{$(m,j)$}};
            \end{tikzpicture}
        \end{center}
        Similarly, column $i$ will refer to the edges between $(i-1,j)$ and $(i,j)$ for $j=0,\dots, n$. There are $m$ columns, numbered from $1$ to $m$.
    \end{notation}

\subsection{Characterizing saturated transfer systems for $C_{p^mq^n}$}
        
        We now characterize the sets of cover relations within saturated transfer systems for $C_{p^mq^n}$. As above, we denote the subgroup $C_{p^iq^j}$ by the pair $(i,j)$. The cover relations in this case are of the form $(i,j) \to (i+1,j)$ and $(i,j) \to (i,j+1)$, that is, the edges of the grid.

        \begin{theorem}
        \label{lemmaX}
            Let $S$ be a set of cover relations within the lattice $[m]\times [n]\cong \Sub(C_{p^mq^n})$. Then $S$
            is the set of all cover relations within a saturated transfer system if and only
            if the following conditions are satisfied:
            \begin{enumerate}
             \item if $(i,j) \to (i+1,j)$ is in $S$, then $(i,k) \to (i+1,k)$ for all $k<j$;
             \item if $(i,j) \to (i,j+1)$ is in $S$, then $(k,j) \to (k,j+1)$ for all $k<i$;
             \item if three out of the four edges in the square
             
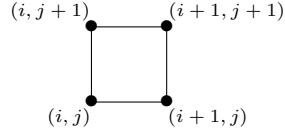
\begin{center}
                 \begin{tikzpicture}
        \node at (0,0) {$\bullet$};
        \node at (-0.3, -0.2) {\tiny{$(i,j)$}};
        \node at (0,1) {$\bullet$};
        \node at (-0.55, 1.2) {\tiny{$(i,j+1)$}};
        \node at (1,0) {$\bullet$};
        \node at (1.55, -0.2) {\tiny{$(i+1,j)$}};
        \node at (1,1) {$\bullet$};
        \node at (1.8, 1.2) {\tiny{$(i+1,j+1)$}};
        \draw (0,0) -- (0,1);
        \draw (0,0) -- (1,0);
        \draw (0,1) -- (1,1);
        \draw (1,0) -- (1,1);
    \end{tikzpicture}
    		\end{center}
             are in $S$, then so is the fourth.
            \end{enumerate}
           
        \end{theorem}

        \begin{proof}
        We first prove the forward direction. Let $\ff = {\to}$ be a saturated transfer system and let $S$ be its set of cover relations.  Conditions (1) and (2) follow from the restriction axiom on transfer systems. For condition (3), suppose that $S$ contains three edges of the square. Then $S$ either contains $(i,j) \to (i+1,j)$ and $(i+1,j) \to  (i+1,j+1)$, or $(i,j) \to (i,j+1)$ and $(i,j+1) \to  (i+1,j+1)$. In either case, by transitivity, $\ff$ must contain $(i,j)\to (i+1,j+1)$. Thus, by restriction and saturation, $\ff$, and hence $S$, must contain all the edges of the square.

To prove the backwards direction, let $S$ be a set of cover relations satisfying the conditions, and consider the transfer system $\ff=\langle S \rangle$ it generates. We will prove that $\ff$ is saturated, and that $S$ is precisely the set of cover relations within $\ff$. 

Recall that $\ff$ is constructed in general by first closing $S$ under restriction, and then closing under transitivity. Conditions (1) and (2) imply that ignoring identities, $S$ itself is already closed under restriction. Indeed, suppose $(i,j) \to (i+1,j)$ is in $S$. Then restriction with respect to $(a,b)$ for $a\leq i+1$ and $b \leq j$ gives $(i,b) \to (i+1,b)$ if $a=i+1$, and gives $(a,b)\to (a,b)$ otherwise. A similar consideration follows for vertical edges. Thus $\ff$ is constructed by closing $S$ under transitivity, and hence, the set of cover relations in $\ff$ is precisely $S$.

To prove $\ff$ is saturated, it suffices to prove that if the edge $(i,j)\to (i+u,j+v)$ is in $\ff$, then all the edges of the corresponding $u\times v$ grid are in $S$. We proceed by induction on $(u,v)$, with the base case $(0,0)$ being trivially satisfied. Suppose the statement is true for $(u,v-1)$ and $(u-1,v)$, unless $v=0$ or $u=0$, in which case we only assume the one that makes sense. If $(i,j)\to (i+u,j+v)$ is in $\ff$, then there is a path of cover relations in $S$ that starts at $(i,j)$ and ends at $(i+u,j+v)$. Assume without loss of generality that the last step of the path is the horizontal edge $(i+u-1,j+v) \to (i+u,j+v)$. 
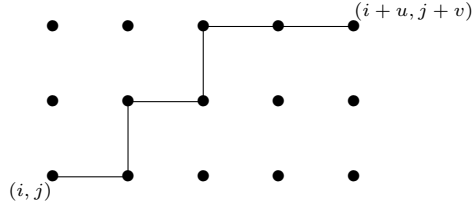
 \begin{center}
        \begin{tikzpicture}
            \node at (0,0) {$\bullet$};
            \node at (-.3,-.2) {\tiny{$(i,j)$}};
            \node at (1,0) {$\bullet$};
            \node at (2,0) {$\bullet$};
            \node at (3,0) {$\bullet$};
            \node at (4,0) {$\bullet$};
            \node at (0,1) {$\bullet$};
            \node at (1,1) {$\bullet$};
            \node at (2,1) {$\bullet$};
            \node at (3,1) {$\bullet$};
            \node at (4,1) {$\bullet$};
            \node at (1,2) {$\bullet$};
            \node at (0,2) {$\bullet$};
            \node at (2,2) {$\bullet$};
            \node at (3,2) {$\bullet$};
            \node at (4,2) {$\bullet$};
            \node at (4.8,2.2) {\tiny{$(i+u,j+v)$}};
            \draw (0,0) -- (1,0);
            \draw (1,0) -- (1,1);
            \draw (1,1) -- (2,1);
            \draw (2,1) -- (2,2);
            \draw (2,2) -- (4,2);
        \end{tikzpicture}
    \end{center}
Then there is a path of cover relations from $(i,j)$ to $(i+u-1,j+v)$, which implies $(i,j) \to (i+u-1,j+v)$ is in $\ff$, and by the inductive hypothesis we have that all the edges of the $(u-1)\times v$ grid are in $S$.  Furthermore, by condition (1), all edges $(i+u-1,k) \to (i+u,k)$ for $k\leq j+v$ are in $S$. Now, the squares in the last column of the grid have three of their edges in $S$, so by condition (3) the fourth edge must be in $S$ as well, thus showing that all the edges of the $u\times v$ grid are in $S$.
\end{proof}
        
   \begin{definition}\label{saturated cover}
   Let $S$ be a subset of cover relations within the lattice $[m]\times [n]$. If $S$ satisfies the conditions of \autoref{lemmaX}, we say that $S$ is a \emph{saturated cover}. We denote the set of saturated covers for $[m]\times [n]$ by $\SCov(m,n)$.
   \end{definition}

\begin{corollary}
There is a bijection between saturated transfer systems on $\STr(C_{p^mq^n})$ and $\SCov(m,n)$.
\end{corollary}

\begin{proof}
This follows from \autoref{generatedcover} and \autoref{lemmaX}.
\end{proof}

 \begin{remark}
                A set $S$ of cover relations is a saturated cover if and only if the restriction of $S$ to each of the $1\times 1$ squares in the grid is a saturated transfer system, i.e., it is one of the following seven options.
            \begin{center}
    \fbox
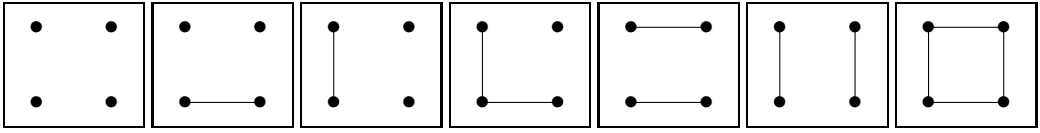
{
    \begin{tikzpicture}
        \node at (0,0) {$\bullet$};
        \node at (0,1) {$\bullet$};
        \node at (1,0) {$\bullet$};
        \node at (1,1) {$\bullet$};
    \end{tikzpicture}
    }
    \fbox{
    \begin{tikzpicture}
        \node at (0,0) {$\bullet$};
        \node at (0,1) {$\bullet$};
        \node at (1,0) {$\bullet$};
        \node at (1,1) {$\bullet$};
        \draw (0,0) -- (1,0);
    \end{tikzpicture}
    }
    \fbox{
    \begin{tikzpicture}
        \node at (0,0) {$\bullet$};
        \node at (0,1) {$\bullet$};
        \node at (1,0) {$\bullet$};
        \node at (1,1) {$\bullet$};
        \draw (0,0) -- (0,1);
    \end{tikzpicture}
    }
    \fbox{
    \begin{tikzpicture}
        \node at (0,0) {$\bullet$};
        \node at (0,1) {$\bullet$};
        \node at (1,0) {$\bullet$};
        \node at (1,1) {$\bullet$};
        \draw (0,0) -- (1,0);
        \draw (0,0) -- (0,1);
    \end{tikzpicture}
    }
    \fbox{
    \begin{tikzpicture}
        \node at (0,0) {$\bullet$};
        \node at (0,1) {$\bullet$};
        \node at (1,0) {$\bullet$};
        \node at (1,1) {$\bullet$};
        \draw (0,0) -- (1,0);
        \draw (0,1) -- (1,1);
    \end{tikzpicture}
    }
    \fbox{
    \begin{tikzpicture}
        \node at (0,0) {$\bullet$};
        \node at (0,1) {$\bullet$};
        \node at (1,0) {$\bullet$};
        \node at (1,1) {$\bullet$};
        \draw (0,0) -- (0,1);
        \draw (1,0) -- (1,1);
    \end{tikzpicture}
    }
    \fbox{
    \begin{tikzpicture}
        \node at (0,0) {$\bullet$};
        \node at (0,1) {$\bullet$};
        \node at (1,0) {$\bullet$};
        \node at (1,1) {$\bullet$};
        \draw (0,0) -- (0,1);
        \draw (0,0) -- (1,0);
        \draw (0,1) -- (1,1);
        \draw (1,0) -- (1,1);
    \end{tikzpicture}
    }
\end{center}
        \end{remark}

Note that by (1) and (2) of \autoref{lemmaX}
it is sufficient to record the highest horizontal edge in each column and the rightmost vertical edge in each row. Equivalently, we can record the number of horizontal edges in each column and the number of vertical edges in each row.
This along with (3) of \autoref{lemmaX} allows us to 
demonstrate a bijection which leads us to 
encode $C_{p^m q^n}$-saturated transfer systems much
more compactly in terms of \emph{compatible codes} as 
follows.

    \begin{definition}
    \label{aibj}
    Let $S$ be a saturated cover on $[m] \times [n]$. For $1\leq i\leq m$, $1\leq j\leq n$,
    \begin{align*}
        a_i &= |\{k\colon (i-1,k)\to (i,k)\in S\}|\\
        b_j &= |\{k \colon (k,j-1)\to (k,j)\in S\}|
    \end{align*}
    In other words, $a_i$ is the number of horizontal edges in column $i$, which can range from $0$ to $n+1$, and $b_j$ is the number of vertical edges in row $j$, which can range from $0$ to $m+1$. We call $a=(a_1,\ldots,a_m)$ the \emph{horizontal code} of $S$, and $b=(b_1,\ldots,b_n)$ the \emph{vertical code} of $S$.
    \end{definition}

    \begin{example}
        Consider the following saturated cover $S$ on $[3]\times [2]$.
        
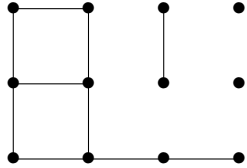
\begin{center}
            \begin{tikzpicture}
                \node at (0,0) {$\bullet$};
                \node at (1,0) {$\bullet$};
                \node at (2,0) {$\bullet$};
                \node at (3,0) {$\bullet$};
                \node at (0,1) {$\bullet$};
                \node at (1,1) {$\bullet$};
                \node at (2,1) {$\bullet$};
                \node at (3,1) {$\bullet$};
                \node at (0,2) {$\bullet$};
                \node at (1,2) {$\bullet$};
                \node at (2,2) {$\bullet$};
                \node at (3,2) {$\bullet$};
                \draw (0,0) -- (0,1);
                \draw (0,2) -- (1,2);
                \draw (0,1) -- (1,1);
                \draw (0,0) -- (1,0);
                \draw (1,0) -- (1,1);
                \draw (1,0) -- (2,0);
                \draw (2,1) -- (2,2);
                \draw (1,1) -- (1,2);
                \draw (0,1) -- (0,2);
                \draw (2,0) -- (3,0);
            \end{tikzpicture}
        \end{center}
        The horizontal code of $S$ is $(3,1,1)$ and the vertical code of $S$ is $(2,3)$.
    \end{example}
    
    \begin{definition}
    Let $m,n\geq 0$. A pair of \emph{compatible codes} on $[m]\times [n]$ is a pair $(a_1,\dots,a_m),(b_1,\dots,b_n)$ of tuples of integers such that $0\le a_i \le n+1$, $0\le b_j\le m+1$, and 
        \[b_{a_i} \leq i \quad \text{and} \quad
         a_{b_j} \leq j,\]
         whenever these are defined.
\end{definition}

\begin{prop}\label{prop:valid_seq}
 There is a bijection between the set of saturated covers and pairs of compatible codes on $[m]\times [n]$.
\end{prop}

\begin{proof}
 Let $(a_1,\dots,a_m),(b_1,\dots,b_n)$ be the horizontal and vertical codes assigned to a saturated cover $S$, fix $i=1,\dots m$, and let $j=a_i$. If $j=0$ or $n+1$, $b_j$ is undefined. Otherwise, consider the following square in the grid.
              
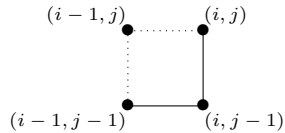
\begin{center}
                 \begin{tikzpicture}
        \node at (0,0) {$\bullet$};
        \node at (-0.8, -0.2) {\tiny{$(i-1,j-1)$}};
        \node at (0,1) {$\bullet$};
        \node at (-0.55, 1.2) {\tiny{$(i-1,j)$}};
        \node at (1,0) {$\bullet$};
        \node at (1.55, -0.2) {\tiny{$(i,j-1)$}};
        \node at (1,1) {$\bullet$};
        \node at (1.3, 1.2) {\tiny{$(i,j)$}};
        \draw[dotted] (0,0) -- (0,1);
        \draw (0,0) -- (1,0);
        \draw[dotted] (0,1) -- (1,1);
        \draw (1,0) -- (1,1);
    \end{tikzpicture}
    		\end{center}
The fact that $a_i=j$ implies that $(i-1,j-1)\to (i,j-1)$ is in $S$, while $(i-1,j)\to (i,j)$ is not, as indicated in the picture. If $b_j>i$, then both vertical edges of the square must be in $S$, violating the 3-out-of-4 condition, thus $b_j\leq i$. The argument for the other inequality is symmetric, thus proving that horizontal and vertical codes are compatible.

Conversely, given compatible codes $(a_1,\dots,a_m),(b_1,\dots,b_m)$, consider the cover relation 
\[ S = \{ (i-1,j) \to (i,j) \mid 1\leq i \leq m \text{ and } j<a_i\} \cup \{ (i,j-1) \to (i,j) \mid 1\leq j \leq n \text{ and } i<b_j\}.\]
By construction $S$ satisfies conditions (1) and (2), and just as above, the inequality constraints imply condition (3).  These two constructions are inverses of each other, thus establishing the bijection.
\end{proof}

\section{Enumeration of saturated transfer systems on $C_{p^mq^n}$}\label{sec:enum}

        The main result of this section is a closed formula for the number of 
        saturated transfer systems on $C_{p^m q^n}$. Throughout we denote the 
        number of saturated transfer systems
        for $C_{p^m q^n}$ as $s(m,n)$. We use the concept of saturated covers of \autoref{saturated cover}.

    \subsection{Recursive formula}\label{sec:recursive}
 We first prove a recursive formula for the number
        of saturated transfer systems. This 
        recursive formula allows us to prove a closed formula for 
        the number of saturated transfer systems that depends on $m$ and $n$
        in \autoref{closed_form_subsection}. 
        
        \begin{theorem}
        \label{recursive}
            Let $m,n\geq 0$. Then
            \begin{equation*}
                s(m,n+1) = s(m,n) + \sum_{k = 0}^{m}{{m+1 \choose k} s(k,n)}.
    \end{equation*}
        \end{theorem}
    
    In essence, we build saturated covers on $[m]\times [n+1]$ out of saturated covers on $[k]\times[n]$, where $k$ ranges from 0
 to $m$. In order to do so, we partition the set of saturated covers on $[m]\times [n+1]$ using the following construction. Let $\cP[m]$ denote the power set of $[m]=\{0,1,\dots, m\}$. For the remainder of the section, we fix $m,n\geq 0$. 

\begin{construction}\label{cons:circling}
We construct a function $c\colon \SCov(m,n)\to \cP[m]$ as follows. Let $S$ be a saturated cover, and suppose $(k,n)\to (k,n+1)$ is the rightmost vertical edge on the $(n+1)$-th row of $S$. That is, $(k,n)\to (k,n+1)$ is in $S$, but  $(k+1,n)\to (k+1,n+1)$ is not. If $S$ has no vertical edges on the $(n+1)$-th row, we set $k=-1$. Then we define
\[c(S)=\{0,1,\dots, k\} \cup \{ i \mid i >k+1 \text{ and } (i-1,n+1) \to (i,n+1) \not\in S\}.
\]
\end{construction}

We will prove \autoref{recursive} by enumerating the fibers of $c$. We first give an example to explain the information encoded by $c$.

\begin{remark}\label{circling}
Let $S$ be a saturated cover on $[4]\times[1]$ with $c(S)=\{0,1,4\}$. The figure below shows $S$, with edges that $S$ must contain in solid black, and the edges $S$ cannot contain in dashed red. The vertices corresponding to the elements in $c(s)$ are marked with a circle.
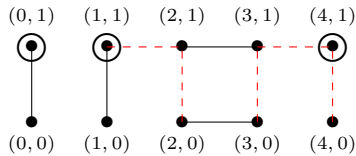
\begin{center}
        \begin{tikzpicture}
            \node at (0,0) {$\bullet$};
            \node at (0,1) {$\bigodot$};
            \node at (0,1) {$\bullet$};
            \node at (1,0) {$\bullet$};
            \node at (1,1) {$\bigodot$};
            \node at (1,1) {$\bullet$};
            \node at (2,0) {$\bullet$};
            \node at (2,1) {$\bullet$};
            \node at (3,0) {$\bullet$};
            \node at (3,1) {$\bullet$};
            \node at (4,0) {$\bullet$};
            \node at (4,1) {$\bullet$};
            \node at (4,1) {$\bigodot$};
            \node at (0,-.3) {\tiny{$(0,0)$}};
            \node at (1,-.3) {\tiny{$(1,0)$}};
            \node at (2,-.3) {\tiny{$(2,0)$}};
            \node at (3,-.3) {\tiny{$(3,0)$}};
            \node at (4,-.3) {\tiny{$(4,0)$}};
            \node at (0,1.4) {\tiny{$(0,1)$}};
            \node at (1,1.4) {\tiny{$(1,1)$}};
            \node at (2,1.4) {\tiny{$(2,1)$}};
            \node at (3,1.4) {\tiny{$(3,1)$}};
            \node at (4,1.4) {\tiny{$(4,1)$}};
            \draw (0,0) -- (0,1);
            \draw (1,0) -- (1,1);
            \draw [red,dashed] (2,0) -- (2,1);
            \draw [red,dashed] (3,0) -- (3,1);
            \draw [red,dashed] (4,0) -- (4,1);
            \draw [red,dashed] (1,1) -- (2,1);
            \draw [red,dashed] (3,1) -- (4,1);
            \draw (2,0) -- (3,0);
            \draw (2,1) -- (3,1);
        \end{tikzpicture}
    \end{center}
Note that the minimal element in the complement of $c(S)$, in this case 2, corresponds to $k+1$ in the construction above, and as such, it is the leftmost vertex without a vertical edge. This explains the dashed red edges $(i,0)\to (i,1)$ for $i=2,3,4$. For $i>k+1$, $(i,n+1)$ is the target of a horizontal edge if and only if $i$ is not in $c(S)$, which explains the solid black edge $(2,1)\to (3,1)$. By the saturated condition, this last edge implies that $S$ must contain $(2,0)\to(3,0)$. Moreover, note that $(1,1)\to (2,1)$ cannot be in $S$; indeed, if it is, then $(1,0)\to (2,0)$ must be in $S$ as well, violating 3-out-of-4. 

The rest of the horizontal edges may or may not be there, with the caveat that the edge $(0,1)\to (1,1)$ is in $S$ if and only if $(0,0) \to (1,0)$ is there too. Thus, the remaining bottom horizontal edges determine $S$.
 \end{remark}

\begin{prop}\label{recursive-bijection}
Let $A\subseteq [m]$. Then 
\[|c^{-1}(A)|=\begin{cases}
			s(|A|,n) &\text{if }A\subsetneq [m],\\
			s(m,n) &\text{if } A=[m].
		\end{cases}\]
\end{prop}

\begin{proof}
 Fix $A\subsetneq [m]$. For notational convenience, let $\ell=|A|$, and let $k+1$ be the minimal element in $[m]\smallsetminus A$. We construct a bijection between $c^{-1}(A)$ and saturated covers on the $[\ell]\times [n]$ grid. 
 
 To  $S \in c^{-1}(A)$ we assign the set of cover relations $S'$ on $[\ell]\times [n]$ obtained by removing the top row and collapsing the columns indexed by all $i >k+1$ in $[m]\smallsetminus A$. As noted in \autoref{circling}, the horizontal $(i-1,n+1) \to (i ,n+1) \in S $, and thus, $(i-1,j)\to (i,j) \in S$ for all $j\in [n]$. By the 3-out-of-4 condition, it follows that the left vertical boundary of the $i$-th column is identical to its right vertical boundary; thus this collapsing operation is well-defined.  Note that we are collapsing precisely $(m+1)-|A|-1$ columns, to obtain a grid with $|A|=\ell$
columns. Moreover, every $1\times 1$ square in $S'$ was a square in $S$, and hence is saturated. This implies $S'$ is saturated.

Conversely, let $T$ be a saturated cover on $[\ell]\times [n]$. We construct a saturated cover $T^*$ on $[m]\times [n+1]$ as follows:
\begin{itemize}
\item For $i>k+1$ in $[m]\smallsetminus A$ and all $j$, include the horizontal edges $(i-1,j)\to (i,j)$.
\item Take $T$ and expand it horizontally, so that the gaps coincide precisely with the columns filled in the previous step. When expanding, the vertical edges are repeated at both ends.
\item For $i=0,\dots, k$, include the vertical edges $(i,n) \to (i,n+1)$.
\item For $i=1,\dots, k$, include the horizontal edge $(i-1,n+1) \to (i,n+1)$ if and only if $(i-1,n) \to (i,n)$ is in $T$.
\end{itemize}
By construction, $T^*$ is a saturated cover,  $c(T^*)=A$, and $(T^*)'=T$. For a saturated cover $S\in c^{-1}(A)$, the considerations in \autoref{circling} show that $(S')^*=S$, thus proving that $(~)'$ and $(~)^*$ are inverse bijections.

In the case that $A=[m]$, if $S\in c^{-1}(A)$, we have that $S$ contains all the vertical edges $(i,n) \to (i,n+1)$ for $i=0,\dots m$. Thus the horizontal edges $(i-1,n+1)\to (i,n+1)$ are determined by $(i-1,n)\to (i,n)$. Thus there is a bijection between $c^{-1}(A)$ and saturated covers on $[m]\times [n]$ obtained by removing the top row.
\end{proof}

We can now prove the recursion stated in \autoref{recursive}:
\[
                s(m,n+1) = s(m,n) + \sum_{k = 0}^{m}{{m+1 \choose k} s(k,n)}.
    \]
\begin{proof}[Proof of \autoref{recursive}]
    We partition the set of saturated covers on $[m]\times[n+1]$ according to the fibers of the function $c$. Since there are ${m+1 \choose k}$ subsets of $[m]$ of cardinality $k$, the formula above follows from \autoref{recursive-bijection}.
\end{proof}

    \subsection{Closed form}
    \label{closed_form_subsection}
        The recursion for the number of saturated transfer 
        systems on $C_{p^m q^n}$ from the previous section allows us to prove a
        closed form for $s(m,n)$. We first recall the following definition.
        
        \begin{definition}
         For $\ell,k \geq 0$, the \emph{Stirling number of the second kind} $\stirling{\ell}{k}$ counts the  number of partitions of a set with $\ell$ elements into $k$ non-empty subsets.
           \end{definition}
\begin{remark}
Stirling numbers of the second kind satisfy the recurrence
\[
  \stirling{\ell+1}{k} = k\stirling{\ell}{k}+\stirling{\ell}{k-1}
\]
for $k>0$ with
\[
  \stirling 00 = 1\qquad\text{and}\qquad \stirling n0 = \stirling 0n = 0
\]
for $n>0$. They are given by the closed formula
\[
  \stirling \ell k = \frac 1{k!}\sum_{i=0}^k (-1)^{k-i}\binom ki i^\ell.
\]
\end{remark}
        
        \begin{theorem}
        \label{closed form}
            For all $m,n\geq 0$, the sequence $s(m,n)$ satisfies
            \[
         s(m,n)=\sum_{j=2}^{m+2}(-1)^{m-j}  \stirling{m+1}{j-1}\frac{j!}{2}j^n.
            \]
        \end{theorem}

        \begin{proof}            We prove this by induction. We begin by proving
            the base case, when $n = 0$. It can be seen that the saturated covers on $[m]\times [0]$ correspond to $m$-long bitstrings (because each column may have 0 or 1 horizontal edges). Therefore, $s(m,0) = 2^m$. We need to prove that
            \[
             s(m,0) =2^m= \sum_{j=2}^{m+2}(-1)^{m-j} \stirling{m+1}{j-1}  \frac{j!}{2}.
            \]
            From \cite[(1.94d)]{stanley}, we know that 
            \begin{equation*}
                x^m = \sum_{k=1}^{m+1}{\stirling{m+1}{k} (x-1)_{k-1}},
            \end{equation*}
            where $(x)_{k} := x(x-1) \dots (x - (k-1))$. In particular, when $x=-2$,
            \[
                (-2)^m=\sum_{k=1}^{m+1}\stirling{m+1}{k}(-3)(-4)\dots (-k-1)=\sum_{k=1}^{m+1}(-1)^{k+1} \stirling{m+1}{k}\frac{(k+1)!}{2},\]
            from which the base case follows. 
            
        For the inductive step we fix $m,n\geq 0$, assume the statement for all $(k,n)$ with $k\leq m$, and prove the case for $(m,n+1)$.  Using the recursive formula from \autoref{recursive} and the inductive hypothesis, and changing the order of summation, we have that
        \[
        s(m,n+1)=\sum_{j=2}^{m+2}\left[ (-1)^{m-j} \stirling{m+1}{j-1}+ \sum_{k=j-2}^{m}(-1)^{k-j}\binom{m+1}{k}\stirling{k+1}{j-1}\right]\frac{j!}{2} j^n.
        \]
        
        Thus, to prove the inductive step it suffices to prove that for all $2\leq j \leq m+2$, 
        \[\stirling{m+1}{j-1} \cdot j = \stirling{m+1}{j-1}+\sum_{k=j-2}^{m}(-1)^{m-k}\binom{m+1}{k}\stirling{k+1}{j-1}.\]
        This follows directly by applying the combinatorial identity in \autoref{key_lemma_closed} below, with $\ell=m+1$ and $r=j-1$.
\end{proof}

\begin{lemma}
        \label{key_lemma_closed}
            Let $\ell, r \in \mathbb{Z}$ such that 
            $0 \leq r \leq \ell$. Then,  
            \[
                (\ell-r) \stirling{\ell}{r}=\sum_{t=1}^{\ell-r}(-1)^{t+1}{\ell \choose t+1}\stirling{\ell-t}{r}
            \]
        \end{lemma}
        
        \begin{proof}
        We prove the above identity via a combinatorial proof. To do so, we consider \emph{marked paritions}. 
    Given two nonnegative integers $\ell$ and $r$ with $r \leq \ell$ a 
    \textit{marked partition} is a pair $(\mathcal{P}, q)$ consisting
    of a partition $\mathcal{P}$ of $\{ 1, \dots, \ell \}$ into $r$ nonempty subsets
    together with a distinguished $q \in \{ 1, \cdots, \ell \}$
    such that $q$ is not the minimum of its subset.
Note that 
    the left hand side of the equation enumerates the number of marked partitions for given $\ell$ and $r$, 
as there are $\stirling{\ell}{r}$ ways to 
    choose the partition $\mathcal{P}$ and $(\ell-r)$ possibilities for 
    $q$.
    
    To prove the identity, we count the set of marked partitions using the inclusion-exclusion principle with the following subsets. For $i<j \leq \ell$, we define the subset of the marked partitions
            \[
                A_{ij} := \{ (\mathcal{P},j) \mid   
                i,j \text{ are in the same subset in } \mathcal{P}\}.
            \]
            The union of the $A_{ij}$ is the set of all marked partitions.  For $i<j$ and $k<l$, $A_{ij}\cap A_{kl}$ is empty unless $j=l$, as the elements marked will be distinct. Note moreover that for any $j\leq \ell$, and distinct $i_1,\dots, i_t < j$, 
    \[
        | A_{i_1 j} \cap \dots \cap A_{i_t j}|
        = 
        \stirling{\ell-t}{r}.
    \]
    Indeed, suppose we have $(\mathcal{P}, j) \in A_{i_1 j} \cap \dots \cap A_{i_t j}$. Then $i_1, \dots, i_t$ are in the 
    same subset as $j$ in the marked partition. Thus, the number of such partitions is equal to partitioning
    the remaining $\ell-t-1$ numbers and the subset $\{i_1,\dots,i_t,j\}$ into $r$ parts, which can be done in 
    $\stirling{\ell-t}{r}$ ways. Moreover, there are ${\ell \choose t+1}$ nonempty $t$-fold 
    intersections, because each nonempty $t$-fold intersection is uniquely determined by choice of $t+1$ numbers $i_1,\dots,i_t,j\leq \ell$. By the principle of inclusion-exclusion the above equation follows.
\end{proof}

\begin{remark}
 Notably, the closed formula for $s(m,n)$ in \autoref{closed form} is not symmetric in $m$ and $n$, although by definition we know that $s(m,n)=s(n,m)$.

 The authors conjectured \autoref{closed form} via the following process. 
 We first observed that for small cases of $m$ and $n$, we could write
 \[s(m,n)=\sum_{j=2}^{m+2} b_{mj}j^n,\]
 for some integers $b_{mj}$ that were independent of $n$. To prove that that is indeed the case, let $B_m(x)$ denote the generating function for $s(m,n)$ as a sequence in $n$.  The recursion in \autoref{recursive} implies that
 \[
  B_m(x) = \frac{2^m+x\sum_{k=0}^{m-1}\binom{m+1}{k}B_k(x)}{1-(m+2)x}.
 \]
 From this, one may prove that $B_m(x)$ is a rational function with denominator $\prod_{j=2}^{m+2}(1-jx)$, whence
 \[
  B_m(x) = \sum_{j=2}^{m+2} \frac{b_{mj}}{1-jx}
 \]
 for some $b_{mj}\in \QQ$.  It follows that
 \[
  s(m,n) = \sum_{j=2}^{m+2}b_{mj}j^n,
 \]
 but we did not succeed in producing the values of $b_{mj}$ via this method.  (The residue method for partial fractions would give the answer if we knew the numerator of $B_m(x)$.) Instead, we computed the values of $b_{mj}$ in a range and guessed that
 \[
  b_{mj} = (-1)^{m-j}\stirling{m+1}{j-1}\frac{j!}{2}
 \]
 via an act of OEIS-enabled perspicacity.
\end{remark} 
    
\subsection{Exponential generating function}    

We now consider the two-variable exponential generating function corresponding to $\{s(m,n)\}$. The content of this section is mostly due to Igor Kriz.

\begin{definition}
Let 
\[f(x,y)=\sum_{m,n\geq 0} \frac{s(m,n)}{m!n!} x^my^n\]
be the exponential generating function corresponding to $\{s(m,n)\}$.
\end{definition}

The recursive formula in \autoref{recursive} allows us to get a closed formula for $f(x,y)$.

\begin{theorem}\label{exp gen}
The exponential generating function for $\{s(m,n)\}$ satisfies
\[f(x,y)=\frac{e^{2x+2y}}{(e^x+e^y-e^{x+y})^3}.\]
\end{theorem}

\begin{proof}
 Using standard techniques for exponential generating functions together with \autoref{recursive} shows that $f$ is a solution to the PDE
 \[\frac{\partial f}{\partial y}=(e^x+1)f+(e^x-1)\frac{\partial f}{\partial x},\]
 subject to the initial conditions
 \[f(x,0)=e^{2x} \quad \text{and} \quad f(0,y)=e^{2y}.\]
 These initial conditions follow from the fact that $s(m,0)=2^m$ and $s(0,n)=2^n$.
 The general solution is of the form
 \[f(x,y)=\frac{\phi(\frac{e^x-1}{e^{x-y}})e^x}{(e^x-1)^2},\]
with $\phi$ an arbitrary function. The initial conditions give the result.
\end{proof}
        
\section{Saturation Conjecture for $C_{pq^n}$}
\label{sec:sat_conj}

In this section we show that the saturation conjecture is true for $G=C_{pq^n}$ for all $n\geq 0$.
    
 We first recall Rubin's characterization of linear isometric transfer systems in the case of finite cyclic groups. Let $k$ be a positive integer, and let $G=C_k$ be the cyclic group of order $k$.

    \begin{definition}\label{def:modular}
  Call $I \subseteq \mathbb{Z}/k\mathbb{Z}$
    that contains $0$ and is closed under additive inverses an \emph{index set}. Given an index set $I$,  
    define the \emph{$I$-modular} transfer system $\mathscr{F}_{I}$, 
    by the following condition: given  $d \mid e \mid k$, 
    \[
       (C_d \rightarrow C_e) \in 
        \mathscr{F}_{I} \iff 
        (I\mod e) + d = (I\mod e).
    \]
    We say $I$ is an \emph{index set} for $\ff_I$.
    \end{definition}

    \begin{prop}[{\cite[Proposition 5.15]{rubin}}]\label{isometriesiffmodular}
      A $G$-transfer system is $I$-modular for some $I$ if and only if it is linear isometric. 
    \end{prop}
    
    The proof is done in two main steps. First, every $G$-universe can be expressed as the direct sum of infinitely many copies of the 2-dimensional representations given by rotation by $2\pi m/k$ for $m \in I$, for some set $I$ that contains 0 is closed under additive inverses. Second, one translates the general characterization of the transfer system associated to a linear isometries operad of \cite[Theorem 4.18]{BH} in terms of this specific decomposition.
    
    \autoref{isometriesiffmodular} implies that to verify the saturation conjecture it is sufficient to build an index set $I$ for each saturated transfer system. Note that the following proposition follows directly from definitions and will allow us to recursively construct index sets via the $\ell=pq^n$, $k=pq^{n+1}$ case.
    
    \begin{prop}\label{easy}
    Let $\ell \mid k$, and suppose that $J$ is an index set for $C_k$. Let $I=(J \mod \ell)$. Then $I$ is an index set for $C_\ell$, and 
    \[\ff_I=(\ff_J)|_{C_\ell}.\]
    \hfill\qedsymbol
    \end{prop}

    Our task now is to  
    recursively generate index sets for 
    saturated transfer systems on $C_{pq^{n+1}}$ from index sets for saturated 
    transfer systems on $C_{pq^n}$. We use the language of saturated covers from \autoref{saturated cover}, and follow the strategy of \autoref{sec:recursive} to split $\SCov(1,n+1)$ into four equivalence classes, based on the fibers of the map $c$ of \autoref{cons:circling}.
    
The four equivalence classes have the following representative top rows. For a saturated cover $S$, the circles denote the vertices in $c(S)$, the edges in solid black are the edges that must be in $S$, and the edges in dashed red are the edges that $S$ doesn't contain. The remaining edges may or may not be in $S$.
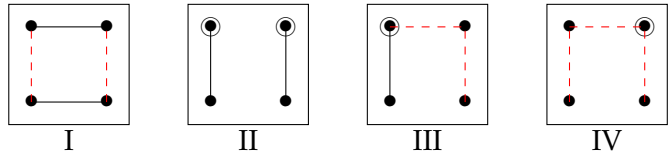
        \begin{center}
        \begin{tikzpicture}
        \node at (0,0) {$\bullet$};
        \node at (0,1) {$\bullet$};
        \node at (1,0) {$\bullet$};
        \node at (1,1) {$\bullet$};
        \draw (0,0) -- (1,0);
        \draw (0,1) -- (1,1);
        \draw [red,dashed] (0,0) -- (0,1);
        \draw [red,dashed] (1,0) -- (1,1);
        \draw (-.3,-.3) rectangle (1.3,1.3);
        \node at (0.5,-.5) {I};
    \end{tikzpicture}\qquad
       \begin{tikzpicture}
        \node at (0,0) {$\bullet$};
        \node at (0,1) {$\bullet$};
        \node at (0,1) {$\odot$};
        \node at (1,0) {$\bullet$};
        \node at (1,1) {$\bullet$};
        \node at (1,1) {$\odot$};
        \draw (0,0) -- (0,1);
        \draw (1,0) -- (1,1);
        \draw (-.3,-.3) rectangle (1.3,1.3);
        \node at (0.5,-.5) {II};
        \end{tikzpicture}\qquad
        \begin{tikzpicture}
        \node at (0,0) {$\bullet$};
        \node at (0,1) {$\bullet$};
        \node at (0,1) {$\odot$};
        \node at (1,0) {$\bullet$};
        \node at (1,1) {$\bullet$};
        \draw (0,0) -- (0,1);
        \draw [red,dashed] (0,1) -- (1,1);
        \draw [red,dashed] (1,0) -- (1,1);
        \draw (-.3,-.3) rectangle (1.3,1.3);
        \node at (0.5,-.5) {III};
    \end{tikzpicture}\qquad
    \begin{tikzpicture}
        \node at (0,0) {$\bullet$};
        \node at (0,1) {$\bullet$};
        \node at (1,0) {$\bullet$};
        \node at (1,1) {$\bullet$};
        \node at (1,1) {$\odot$};
        \draw [red,dashed] (0,0) -- (0,1);
        \draw [red,dashed] (1,0) -- (1,1);
        \draw [red,dashed] (0,1) -- (1,1);
        \draw (-.3,-.3) rectangle (1.3,1.3);
        \node at (0.5,-.5) {IV};
    \end{tikzpicture}

    \end{center}
    As explained in \autoref{sec:recursive},  within class I, a saturated cover is determined by its restriction to $[0]\times [n]$, while for the other three it is determined by its restriction to $[1]\times [n]$. 

In order for the inductive step to work in all cases, we need to prove the following stronger statement.

\begin{theorem}\label{Cpqn}
Suppose $p,q$ are primes greater than 3 and $n\geq 0$, and let $\ff$ be a saturated transfer system on $C_{pq^n}$. Then there exists an index set $I\subseteq \ZZ/pq^n\ZZ$  such that $\ff=\ff_I$ and $I$ contains a nonzero multiple of $q^n$.
\end{theorem}

\begin{proof}
 We first prove the statement for type I directly, and we then prove the cases for types II, III and IV by induction. The strategy for all cases is to take a saturated transfer system $\ff$, restrict it to a certain subgroup, take an index set for the restriction, and construct an index set for $\ff$ based on the index set for the restriction. We will write $(i,j)$ for the subgroup $C_{p^iq^j}\le C_{pq^n}$ when it is convenient.
 
 \textbf{Type I:} Let $\ff$ be a saturated transfer system on $C_{pq^n}$ of type I, and consider its restriction $\ff \mid _{C_{q^n}}$. By \cite[Theorem 5.18]{rubin}, this saturated transfer system on $C_{q^{n}}$ is induced by an index set $I\subseteq \ZZ/{q^n}\ZZ$.  Set
\[
  J := \{\alpha q^{n}+i\mid 0\le \alpha<p,~i\in I\}.
\]
By construction we have that $J\mod{q^n} = I$, so \autoref{easy} implies that the restrictions of $\ff_J$ and $\ff$ to $C_{q^n}$ coincide. By construction again we have that $J+q^n=J$, thus showing that $(0,n) \to (1,n) \in \ff_J$. By the conditions on saturation these two facts imply that $\ff_J=\ff$. Moreover, taking $\alpha=1$ shows that $q^n \in J$.

For the remainder of the proof, we proceed by induction. For the base case, when $n=0$, we are considering the two saturated systems for $C_p$: the trivial and the complete one (which is of type I). The trivial one is induced by $I=\{0,1,p-1\}$ as long as $p>3$ and the complete one is induced by $I=\{0,1,\dots,p-1\}$. Note that the latter index set is the one we obtain from the direct proof above. 

Our inductive hypothesis is that the statement of the theorem holds for some $n\geq 0$. Let $\ff$ be a saturated transfer system of type II, III or IV on $C_{pq^{n+1}}$, and consider its restriction to $C_{pq^n}$. The inductive hypothesis implies that there exists and index set $I\subseteq \ZZ/pq^n\ZZ$ containing $aq^n$ for some $0<a<p$, and such that
\[ \ff_I = \ff \mid_{C_{pq^n}}.\]
In all three cases, we will produce an index set $J\subseteq \ZZ/pq^{n+1}\ZZ$ such that $J \mod{pq^n}=I$, so that by \autoref{easy} we get that $\ff_J$ and $\ff$ coincide in the restriction to $C_{pq^n}$. We will then show that $J$ does the right thing in the top square (depending on the type) and contains a nonzero multiple of $q^{n+1}$.

\textbf{Type II:} Suppose $\ff$ has type II and take $I\subseteq \ZZ/pq^n\ZZ$ as described above. Set
\[
  J := \{\alpha pq^n+i\mid 0\le \alpha<q,~i\in I\}.
\]
In a manner similar to the type I argument, the reader may verify that $\ff_J=\ff$. To find a nonzero multiple of $q^{n+1}$ in $J$, take $\alpha$ with residue class $-ap^{-1}\in \ZZ/q\ZZ^\times$ and $i=aq^n$.

\textbf{Type III:} Suppose $\ff$ has type III and take $I\subseteq \ZZ/pq^n\ZZ$ as described above. We need to produce an index set $J\subseteq \ZZ/pq^{n+1}\ZZ$ such that $J\mod{pq^n} = I$, some $0\ne bq^{n+1}\in J$, and --- since $\ff$ has type III --- such that $(1,n)\to (1,n+1)\not\in \ff_J$ and $(0,n)\to (0,n+1)\in \ff_J$. By the saturation axioms, these conditions are enough to ensure $\ff_J=\ff$.

Set
\[
  J' := \{\alpha pq^n+i\mid 0\le \alpha <q,~i\in I\}
\]
and
\[
  J := J'\smallsetminus \{aq^n,pq^{n+1}-aq^n\}.
\]
Then $J$ is an index set and it is clear that $J\mod{pq^n}\subseteq I$ with only $aq^n$ and $-aq^n$ possibly in the set difference. For $q>2$ and $\alpha = 1$, the element $pq^n+aq^n$ is in $J$ and reduces to $aq^n\mod{pq^n}$. Similarly, the mod $pq^{n+1}$ negative of this element is in $J$ and reduces to $-aq^n\mod{pq^n}$.

We now check that $(0,n)\to (0,n+1)\in \ff_J$, which amounts to $(J\mod q^{n+1})+q^n = J\mod q^{n+1}$. We claim that $J\mod{q^{n+1}} = J'\mod{q^{n+1}}$, which suffices for this result.  To verify the claim, take $\alpha$ that reduces to $ap^{-1}\mod q$. Then
\[
  aq^n\equiv \alpha pq^n\mod{q^{n+1}},
\]
so $aq^n \in (J \mod{q^{n+1}})$ as needed.

To show that $(1,n)\to (1,n+1)\not\in \ff_J$, we must verify that
\[
  J+pq^n \ne J.
\]
Note that $aq^n\not\in J$, but $(q-1)pq^n+aq^n\in J$, so $aq^n\in J+pq^n$.

Finally, we need to show that $J$ contains some nonzero multiple $bq^{n+1}$ of $q^{n+1}$.  If $a\ne q$, then $J$ contains
\[
  \alpha pq^n+aq^n = (\alpha p+a)q^n
\]
which is divisible by $q^{n+1}$ for some $0<\alpha<q$.  If $a=q$ (which is possible when $q<p$), then we actually must modify the definition of $J$, setting
\[
  J := J' \smallsetminus \{aq^n+pq^n,pq^{n+1}-(aq^n+pq^n)\}.
\]
The above argument still goes through and we can then check that some $0\ne bq^{n+1}\in J$. 

\textbf{Type IV:} Suppose $\ff$ has type IV and take $I\subseteq \ZZ/pq^n\ZZ$ as described above. We construct an index set $J\subseteq \ZZ/pq^{n+1}\ZZ$ containing a nonzero multiple of $q^{n+1}$ such that $J \mod{pq^n} = I$, $(0,n)\to (0,n+1)\not\in \ff_J$ and $(0,n+1)\to (1,n+1)\not\in \ff_J$. The fact that $\ff$ is of type IV and the saturation axioms imply then that $\ff_J=\ff$.

Let $i\in I\smallsetminus 0$. Then by \autoref{alpha} (stated and proved immediately after this proof), there exists $0\leq\alpha_i<q$ such that $\alpha_i pq^n +i \mod{q^{n+1}}$ lies in the interval $[0,q^n)$. By Sunzi's theorem, there exists $0<c<pq$ such that $c$ is a multiple of $q$ and $c\equiv a \mod{p}$. We set
\[J := \{0, cq^n,pq^{n+1}-cq^n\} \cup \left\{ \alpha_i pq^n +i, pq^{n+1}-(\alpha_i pq^n +i) \mid i \in I \smallsetminus \{0,aq^n,pq^n-aq^n\} \right\}.\]
Since $c$ is a nonzero multiple of $q$, we know $cq^n$ is a nonzero multiple of $q^{n+1}$. The other condition on $c$ implies $cq^n\equiv aq^n\mod{pq^n}$, and we thus have that $J \mod{pq^n}=I$. 

To prove that $(0,n)\to (0,n+1)\not\in \ff_J$ we need to check that
\[ (J \mod{q^{n+1}}) + q^n \neq (J \mod{q^{n+1}}).\]
We have that
 \[(J \mod{q^{n+1}}) =\{0\} \cup \{ \alpha_i pq^n +i, q^{n+1}-(\alpha_i pq^n +i) \mid i \in I \smallsetminus \{0,aq^n,pq^n-aq^n\} \}.\]
When considering this set in terms of representatives $\{0,1,\dots,q^{n+1}-1\}$, its elements are concentrated in the intervals $[0,q^n-1]$ and $[q^{n+1}-q^n+1,q^{n+1}-1]$. Basic arithmetic shows that if $q>2$, the translation by $q^n$ of any element in the first interval does not land in either of the two intervals, showing our result.

A similar argument using that $p,q>3$ can be used to prove that $J+q^{n+1} \neq J$, thus showing that $(0,n+1)\to (1,n+1)\not\in \ff_J$, as needed. This finishes the proof.
\end{proof}

\begin{lemma}\label{alpha}
Let $i$ be an integer such that $0<i<pq^{n}$. Then there exists $0\leq\alpha <q$ such that the residue of $\alpha pq^n +i \mod{q^{n+1}}$ lies in the interval $[0,q^n)$.
\end{lemma}

\begin{proof}
 Let $r$ be the residue class of $i \mod{q^n}$. Thus, $0\leq r < q^n$, and there exists $0\leq k <p$ such that $i=kq^n+r$. By B\'{e}zout's identity, there exist $c,d\in \ZZ$ such that $cp+dq=1$. Let $\beta = -ck \in \ZZ$. Then basic arithmetic shows that 
 \[\beta pq^n + i \equiv r \mod{q^{n+1}}.\]
 Finally, letting $\alpha$ be the residue of $\beta \mod{q}$ achieves the result.
\end{proof}

\bibliography{mybibliography}

\newcommand{\etalchar}[1]{$^{#1}$}
\begin{thebibliography}{LMSM86}

\bibitem[BBR]{BBR}
Scott Balchin, David Barnes, and Constanze Roitzheim.
\newblock ${N}_\infty$-operads and associahedra.
\newblock arXiv:1905.03797v2. To appear in \emph{Pacific Journal of
  Mathematics}.

\bibitem[BH15]{BH}
Andrew~J. Blumberg and Michael~A. Hill.
\newblock Operadic multiplications in equivariant spectra, norms, and
  transfers.
\newblock {\em Adv. Math.}, 285:658--708, 2015.

\bibitem[BOOR]{boor}
Scott Balchin, Kyle Ormsby, Ang{\'e}lica Osorno, and Constanze Roitzheim.
\newblock Model structures on finite total orders.
\newblock In preparation.

\bibitem[BP21]{BP}
Peter Bonventre and Lu\'{\i}s~A. Pereira.
\newblock Genuine equivariant operads.
\newblock {\em Adv. Math.}, 381:107502, 133, 2021.

\bibitem[FOO{\etalchar{+}}]{selfdual}
Evan~E. Franchere, Kyle Ormsby, Ang\'{e}lica~M Osorno, Weihang Qin, and Riley
  Waugh.
\newblock Self-duality of the lattice of transfer systems via weak
  factorization systems.
\newblock arXiv:2102.04415v2. To appear in \emph{Homology, Homotopy and
  Applications.}

\bibitem[GW18]{GW}
Javier~J. Guti\'{e}rrez and David White.
\newblock Encoding equivariant commutativity via operads.
\newblock {\em Algebr. Geom. Topol.}, 18(5):2919--2962, 2018.

\bibitem[LMSM86]{LMS}
L.~G. Lewis, Jr., J.~P. May, M.~Steinberger, and J.~E. McClure.
\newblock {\em Equivariant stable homotopy theory}, volume 1213 of {\em Lecture
  Notes in Mathematics}.
\newblock Springer-Verlag, Berlin, 1986.
\newblock With contributions by J. E. McClure.

\bibitem[Rub]{rubin_comb}
Jonathan Rubin.
\newblock Combinatorial {$N_\infty$} operads.
\newblock arXiv:1705.03585v3.

\bibitem[Rub20]{rubin}
Jonathan Rubin.
\newblock Detecting {S}teiner and linear isometries operads.
\newblock {\em Glasgow Mathematical Journal}, pages 1--36, 05 2020.

\bibitem[Sta12]{stanley}
Richard~P. Stanley.
\newblock {\em Enumerative combinatorics. {V}olume 1}, volume~49 of {\em
  Cambridge Studies in Advanced Mathematics}.
\newblock Cambridge University Press, Cambridge, second edition, 2012.

\end{thebibliography}
\bibliographystyle{alpha}
\end{document}